\pdfoutput=1
\RequirePackage{ifpdf}
\ifpdf 
\documentclass[pdftex]{sigma}
\else
\documentclass{sigma}
\fi

\usepackage{stmaryrd}
\usepackage{tikz}
\usepackage[all]{xy}

\CompileMatrices

\newtheorem{thm}{Theorem}[section]
\newtheorem{lem}[thm]{Lemma}
\newtheorem{prop}[thm]{Proposition}

\theoremstyle{definition}
\newtheorem{defi}[thm]{Definition}
\newtheorem{constr}[thm]{Construction}
\newtheorem{rem}[thm]{Remark}

\def\CC{{\mathbb C}}
\def\TT{{\mathbb T}}
\def\ZZ{{\mathbb Z}}
\def\RR{{\mathbb R}}
\def\QQ{{\mathbb Q}}
\def\PP{{\mathbb P}}

\DeclareMathOperator{\Cl}{Cl}
\DeclareMathOperator{\cl}{cl}
\DeclareMathOperator{\cone}{cone}
\DeclareMathOperator{\Spec}{Spec}

\def\bangle#1{{\langle #1 \rangle}}

\begin{document}

\newcommand{\arXivNumber}{2405.04862}

\renewcommand{\PaperNumber}{042}

\FirstPageHeading

\ShortArticleName{On Degenerations of the Projective Plane}

\ArticleName{On Degenerations of the Projective Plane}

\Author{J\"urgen HAUSEN~$^{\rm a}$, Katharina KIR\'{A}LY~$^{\rm a}$ and Milena WROBEL~$^{\rm b}$}

\AuthorNameForHeading{J.~Hausen, K.~Kir\'{a}ly and M.~Wrobel}

\Address{$^{\rm a)}$~Mathematisches Institut, Universit\"at T\"ubingen,\\
\hphantom{$^{\rm a)}$}~Auf der Morgenstelle 10, 72076 T\"ubingen, Germany}
\EmailD{\href{mailto:juergen.hausen@uni-tuebingen.de}{juergen.hausen@uni-tuebingen.de}, \href{mailto:kaki@math.uni-tuebingen.de}{kaki@math.uni-tuebingen.de}}

\Address{$^{\rm b)}$~Institut f\"ur Mathematik, Universit\"at Oldenburg,
26111 Oldenburg, Germany}
\EmailD{\href{mailto:milena.wrobel@uni-oldenburg.de}{milena.wrobel@uni-oldenburg.de}}

\ArticleDates{Received December 17, 2024, in final form June 07, 2025; Published online June 12, 2025}

\Abstract{Complementing results of Hacking and Prokhorov, we determine in an explicit manner all log terminal, rational, degenerations of the projective plane that allow a non-trivial torus action.}


\Classification{14L30; 14J26; 14J10; 14D06}

\section{Introduction}

The aim of this note is to determine explicitly
all log terminal, rational degenerations of the~projective plane $\PP_2$
that admit a non-trivial torus action, see
Theorem~\ref{cor:planedegs}; note that log
terminality for a surface merely means to
have at most quotient singularities.
Recall from~\cite{Ma} that a degeneration of $\PP_2$ is the
central fiber $X_0$ of a proper flat analytic family of
surfaces over the unit disk such that $X_t \cong \PP_2$
for all $t \ne 0$.
Manetti~\cite{Ma} characterized the log terminal
degenerations of $\PP_2$
as the projective algebraic complex surfaces of
Picard number one with vanishing plurigenera having
at most singularities of the type $\frac{1}{n^2}(1,na-1)$
with coprime $a,n \in \ZZ_{>0}$, that means quotients
of $0 \in \CC^2$ by the linear action of a cyclic group
of order $n^2$ with the weights $(1,na-1)$.

The description of all normal degenerations of the projective
plane involves the \emph{Markov numbers}.
These are by definition the entries of the \emph{Markov triples}
which in turn are the triples $(x,y,z)$ of positive
integers satisfying the diophantic equation
\[
x^2 + y^2 + z ^2 = 3xyz.
\]
From any Markov triple $(x,y,z)$, one obtains
new ones via mutations, that means by permuting
its entries and passing to $(x,y,3xy-z)$.
The vertices of the \emph{Markov tree} are the
normalized (i.e., ascendingly ordered)
Markov triples, representing all triples obtained
from a given one by permuting its entries.
In the Markov tree, two triple classes are
\emph{adjacent}, that means joined by an edge,
if and only if they are distinct and arise
from each other by a mutation.
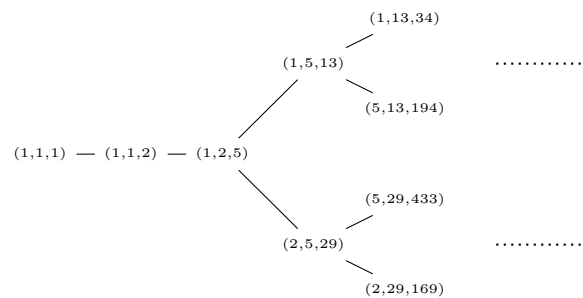
\[
\begin{tikzpicture}[scale=0.6]
\sffamily
\node[] (111) at (0,0) {$\scriptscriptstyle (1,1,1)$};
\node[] (112) at (2,0) {$\scriptscriptstyle (1,1,2)$};
\node[] (125) at (4,0) {$\scriptscriptstyle (1,2,5)$};
\node[] (1513) at (6,2) {$\scriptscriptstyle (1,5,13)$};
\node[] (2529) at (6,-2) {$\scriptscriptstyle (2,5,29)$};
\node[] (113134) at (8,3) {$\scriptscriptstyle (1,13,34)$};
\node[] (513194) at (8,1) {$\scriptscriptstyle (5,13,194)$};
\node[] (529533) at (8,-1) {$\scriptscriptstyle (5,29,433)$};
\node[] (229169) at (8,-3) {$\scriptscriptstyle (2,29,169)$};
\draw[] (111) edge (112);
\draw[] (112) edge (125);
\draw[] (125) edge (1513);
\draw[] (125) edge (2529);
\draw[] (1513) edge (113134);
\draw[] (1513) edge (513194);
\draw[] (2529) edge (529533);
\draw[] (2529) edge (229169);
\draw[thick, dotted] (10,2) edge (12,2);
\draw[thick, dotted] (10,-2) edge (12,-2);
\end{tikzpicture}
\]

The Markov tree admits an interpretation in terms of rational
projective complex surfaces.
With any vertex, represented by a Markov triple $(k_1,k_2,l)$,
Hacking and Prokhorov associate the weighted projective plane
\smash{$\PP_{(k_1^2,k_2^2,l^2)}$} and they show that these are in fact
all toric normal degenerations of $\PP_2$,
see~\cite[Corollary~1.2 and Theorem~4.1]{HaPro}.
The edges, joining adjacent Markov triples $(k_1,k_2,l_1)$
and $(k_1,k_2,l_2)$,
are reflected combinatorially by mutations of the simplices
associated with the two weighted projective
planes~\cite{AkGaEtAl,AkKa}
and geometrically by a flat one parameter family having
both of them as fibers~\cite{Hack,Il,Pe};
moreover,~\cite{UrZu} studies the birational geometry
behind the Markov tree.
For an explicit geometric realization of a given edge joining
adjacent Markov triples $(k_1,k_2,l_1)$ and $(k_1,k_2,l_2)$,
one associates with it the following surface,
living in a~weighted projective space, see
also~\cite[Example~7.7]{Hack}:
\[
X(k_1,k_2,l_1,l_2)
 :=
V\bigl(T_1T_2 + T_3^{l_1} + T_4^{l_2}\bigr)
 \subseteq
\PP_{(k_1^2,k_2^2,l_2,l_1)}.
\]
By Theorem~\ref{prop:markovsurface}, the surface
$X(k_1,k_2,l_1,l_2)$ is well defined, quasismooth
(i.e., has at most cyclic quotient singularities),
rational, del Pezzo, of Picard number
one and it comes with an effective $\CC^*$-action.
Moreover, we obtain the following geometric connections
between $X(k_1,k_2,l_1,l_2)$ and the weighted projective
planes \smash{$\PP_{(k_1^2,k_2^2,l_i^2)}$} given by the two
adjacent triples, see also the (more general)
Construction~\ref{constr:covs} and Propositions~\ref{prop:cstarsurf} and
\ref{prop:degprops}.

\begin{thm}
Let $(k_1,k_2,l_1)$ and $(k_1,k_2,l_2)$ be adjacent Markov
triples and consider the surface $X(k_1,k_2,l_1,l_2)$.
Then we have a commutative diagram
\[
\xymatrix{
&&
{\PP_{(k_1^2, k_2^2, l_2, l_1)}}
\ar@{-->}[dll]_{[z_1,z_2,z_4^{l_1}] \mapsfrom [z_1,z_2,z_3,z_4]\qquad\qquad}
\ar@{-->}[drr]^{\qquad\qquad[z_1,z_2,z_3,z_4] \mapsto [z_1,z_2,z_3^{l_2}]}
&&
\\
{\PP_{(k_1^2,k_2^2,l_1^2)}}
&&
X(k_1,k_2,l_1,l_2)
\ar[u]
\ar[ll]^{1:l_1^2 \quad}
\ar[rr]_{\quad l_2^2:1}
&&
{\PP_{(k_1^2,k_2^2,l_2^2)}}
}
\]
with finite coverings
\smash{$X(k_1,k_2,l_1,l_2) \to {\PP_{(k_1^2,k_2^2,l_i^2)}}$}
of degree $l_i^2$, respectively.
Moreover, there are flat families
$\psi_i \colon \mathcal{X}_i \to \CC$,
where
\begin{gather*}
\mathcal{X}_1
 :=
V\bigl(T_1T_2 + ST_3^{l_1} + T_4^{l_2}\bigr)
 \subseteq
\PP_{(k_1^2, k_2^2, l_2, l_1)} \times \CC,
\\
\mathcal{X}_2
 :=
V\bigl(T_1T_2 + T_3^{l_1} + ST_4^{l_2}\bigr)
 \subseteq
\PP_{(k_1^2, k_2^2, l_2, l_1)} \times \CC
\end{gather*}
and the $\psi_i$ are obtained by restricting the projection
\smash{$\PP_{(k_1^2, k_2^2, l_2, l_1)} \times \CC \to \CC$}.
For the fibers of these families, we have
\[
\psi_i^{-1}(s) \cong X(k_1,k_2,l_1,l_2), \quad s \in \CC^*,
\qquad
\psi_i^{-1}(0) \cong {\PP_{(k_1^2,k_2^2,l_i^2)}}.
\]
\end{thm}

Whereas the degenerations $\mathcal{X}_i \to \CC$
are, as mentioned, well known, the coverings haven't
been observed so far to our knowledge.
Note that there may occur adjacent Markov triples
$(k_1,k_2,l_i)$ with $l_i=1$ for one of the $i$.
This happens if and only if $k_1$ and $k_2$ are
Fibonacci numbers of subsequent odd indices.
In this case, the covering is in fact an isomorphism:
\[
X(k_1,k_2,l_1,l_2) \cong \PP_{(k_1^2,k_2^2,1)}.
\]
If both $l_i$ differ from one, $X(k_1,k_2,l_1,l_2)$
is non-toric.
The \emph{Fibonacci branch} of the Markov tree
hosts the vertices \smash{$\PP_{(1,k_2^2,l^2)}$}
and the edges $X(1,k_2,l_1,l_2)$, where the latter
surfaces also showed~up in~\cite[Remark~6.6]{LiXuZh}.
Our first main result characterizes the surfaces
$X(k_1,k_2,l_1,l_2)$ by their geometric
properties and shows in particular that they
are uniquely determined by the underlying weighted
projective planes.

\begin{thm}
Let $X$ be a non-toric, log terminal,
rational, projective $\CC^*$-surface of Picard
number $\rho(X) = 1$.
Then the following statements are equivalent:
\begin{enumerate}\itemsep=0pt
\item[$(i)$]
The surface $X$ is isomorphic to one of the
surfaces $X(k_1,k_2,l_1,l_2)$.
\item[$(ii)$]
The canonical self intersection number of $X$
is given by $\mathcal{K}_X^2 = 9$.
\end{enumerate}
Moreover, if one of these statements holds, then $X$ is
determined up to isomorphy by the numbers~$k_1$ and $k_2$.
\end{thm}

Let us refer to the representatives \smash{$\PP_{(k_1^2,k_2^2,l_i^2)}$}
of the vertices of the Markov tree as the
\emph{toric Markov surfaces}
and to the representatives $X(k_1,k_2,l_1,l_2)$
of the edges as the \emph{Markov $\CC^*$-surfaces}.
Then the second main result of this note,
Theorem~\ref{cor:planedegs}, provides
us with the following statement on degenerations
of the projective plane $\PP_2$ in the sense of
Manetti~\cite[Definition~1]{Ma}.

\goodbreak

\begin{thm}
Let $X$ be a log terminal,
rational, projective surface
with a non-trivial torus action.
Then the following statements are equivalent:
\begin{enumerate}\itemsep=0pt
\item[$(i)$]
$X$ is a degeneration of the projective
plane.
\item[$(ii)$]
We have $\rho(X) = 1$ and $\mathcal{K}_X^2 = 9$.
\item[$(iii)$]
$X$ is a toric Markov surface or a
Markov $\CC^*$-surface.
\end{enumerate}
\end{thm}

As a consequence of this characterization,
given any log terminal, rational degeneration $X$
of the projective plane with precisely one
singularity $x \in X$, we can say the following:
if the local Gorenstein index of $x \in X$ is a
Markov but not a Fibonacci number, then $X$ does
not allow a~non-trivial $\CC^*$-action.

Moreover, let us relate the second main to
\emph{test configurations} of the projective plane.
Roughly speaking, these are $\CC^*$-equivariant
flat families $\psi \colon \mathcal{X} \to \CC$
with general fiber $\psi^{-1}(1) \cong \PP_2$ such
that the $\CC^*$-action on $\CC$ is given by the
multiplication, see~\cite[Definition~A.1]{AsBeEtAl}.
According to~\mbox{\cite[Theorem~A.3]{AsBeEtAl}}, the surfaces
which arise as the central fiber $\psi^{-1}(0)$ of a test
configuration of $\PP_2$ are given up to isomorphy
by the vertices of the Fibonacci branch of
the Markov tree and the edges touching one of the latter
vertices.
All the other vertices and edges from the Markov tree
represent degenerations of the projective plane in the
sense of~\cite[Definition~1]{Ma} which do not occur as the
central fiber of any test configuration of the
projective plane.


\section{Toric Markov surfaces}

We provide the necessary facts on toric Markov surfaces,
that means weighted projective planes given by a squared
Markov triple.
The main observation of the section,
Proposition~\ref{prop:toricmarkovchar},
characterizes the toric Markov surfaces as the
toric surfaces of Picard number one and canonical
self intersection nine.
The reader is assumed to be familiar with the basics of
toric geometry; we refer to~\cite{CoLiSc} for the
background.

\begin{constr}[fake weighted projective spaces as toric varieties]
\label{constr:fwps}
Consider an $n \times (n+1)$ \emph{generator matrix},
that means an integral matrix
\[
P = \begin{bmatrix} v_0 & \dots & v_n \end{bmatrix}
\]
the columns $v_i \in \ZZ^n$ of which are parwise distinct,
primitive and generate $\QQ^n$ as a convex cone.
For each $i = 0, \dots, n$, we obtain a convex,
polyhedral cone
\[
\sigma_i := \cone(v_j; j = 0, \dots, n, j \ne i).
\]
These $\sigma_i$ are the maximal cones of a
fan $\Sigma = \Sigma(P)$ in $\ZZ^n$.
The associated toric variety $Z=Z(P)$ is an
$n$-dimensional \emph{fake weighted projective space}.
\end{constr}

The fake weighted projective spaces turn out to be
precisely the $\QQ$-factorial projective toric
varieties of Picard number one.
In particular, the fake weighted projective
planes are exactly the projective toric surfaces
of Picard number one.
We will also benefit from the following alternative
approach.

\begin{rem}\label{rem:fwpsasquots}
The fake weighted projective spaces $Z(P)$ are quotients
of $\CC^{n+1} \setminus \{0\}$. The matrix $P=(p_{ij})$
from Construction~\ref{constr:fwps} defines a homomorphism
\[
p \colon \ \TT^{n+1} \to \TT^n,
\qquad
(t_0,\dots,t_n)
 \mapsto
\bigl(t_0^{p_{10}} \cdots t_n^{p_{1n}}, \dots , t_0^{p_{n0}} \cdots t_n^{p_{nn}}\bigr),
\]
the subgroup $H := \ker(p) \subseteq \TT^{n+1}$ is a
direct product $\CC^* \times G$ with a finite
subgroup $G \subseteq \TT^{n+1}$ and for the
induced action of $H$ on $\CC^{n+1}$ we have
\[
Z(P) = \bigl(\CC^{n+1} \setminus\{0\}\bigr) / H.
\]
This provides us with \emph{homogeneous coordinates}
as for the classical projective space:
We write $[z_0,\dots, z_n] \in Z(P)$ for the
$H$-orbit through
$(z_0,\dots, z_n) \in \CC^{n+1} \setminus \{0\}$.
\end{rem}

Recall that for a point~$x$ of a normal variety $X$
the \emph{local class group} is the factor group~$\Cl(X,x)$
of the group of all Weil divisors on $X$
modulo those being principal on some open
neighbourhood of~$x$. We denote by $\cl(x)$ the
order of $\Cl(X,x)$.

\begin{rem}
\label{rem:fakewv}
Consider $P = [v_0, \dots, v_n]$ as in
Construction~\ref{constr:fwps} and
the associated $Z=Z(P)$.
The \emph{fake weight vector} associated
with $P$ is
\[
w
=
w(P)
=
(w_0,\dots,w_n)
\in
\ZZ_{>0}^{n+1},
\qquad
w_i := \vert {\det}(v_j; j = 0,\dots,n, j \ne i) \vert.
\]
For the divisor class group and the local class groups
of the toric fixed points $z(i)$, having $i$-th
homogeneous coordinate one and all others zero,
we obtain
\[
\Cl(Z) = \ZZ^n / \mathrm{im}(P^*) \cong \ZZ \oplus \Gamma,
\qquad
\vert \Gamma \vert = \gcd(w_0,\dots,w_n),
\qquad
\cl(z(i)) = w_i.
\]
Moreover, $\Cl(Z) \cong \ZZ \oplus \Gamma$ can
be identified with the character group of
$H \cong \CC^* \times G$ from
Remark~\ref{rem:fwpsasquots} via the isomorphism
$H \cong \Spec \CC[\ZZ^n / \mathrm{im}(P^*)]$.
\end{rem}

\begin{rem}Let $Z=Z(P)$ arise from Construction~\ref{constr:fwps}
and $w = w(P)$ as in Remark~\ref{rem:fakewv}.
Then $\Cl(Z)$ is torsion free if and
only if $w \in \ZZ^{n+1}$ is primitive.
In the latter case, $Z$ equals the
\emph{weighted projective space}
\smash{$\PP_{(w_0, \dots, w_n)}$}.
\end{rem}

\begin{prop}\label{prop:fwppKsquare}
For a fake weighted projective plane $Z = Z(P)$
with fake weight vector $w = w(P) = (w_0,w_1,w_2)$,
the canonical self intersection number is given by
\[
\mathcal{K}_Z^2
 =
\frac{(w_0+w_1+w_2)^2}{w_0w_1w_2}.
\]
\end{prop}

\begin{proof}
We may assume that the generator matrix $P$
of our fake weighted projective plane $Z$ is
of the form
\[
P
 =
\begin{bmatrix}
l_0 & l_1 & l_2
\\
d_0 & d_1 & d_2
\end{bmatrix},
\qquad
l_i \ne 0, \quad i = 0,1,2.
\]
Then we can use, for instance,~\cite[Remark~3.3]{HaHaSp}
for the computation of the canonical self intersection
number.
\end{proof}

\begin{defi}
By a \emph{toric Markov surface} we mean a surface
isomorphic to a weighted projective plane
$\PP\bigl(k_0^2,k_1^2,k_2^2\bigr)$, where $(k_0,k_1,k_2)$ is
a Markov triple.
\end{defi}

\begin{prop}
\label{prop:toricmarkovchar}
Let $Z$ be a projective toric surface of
Picard number one.
Then the following statements are equivalent:
\begin{enumerate}\itemsep=0pt
\item[$(i)$]
$Z$ is a toric Markov surface.
\item[$(ii)$]
We have $\mathcal{K}_Z^2 = 9$.
\end{enumerate}
\end{prop}

The proof relies on the following (known)
elementary statement on the positive integer
solutions of the ``squared Markov identity'';
see also~\cite{GyMa} for recent, further going
work in that direction.

\goodbreak

\begin{lem}
\label{lem:diophantic}
The positive integer solutions of $(w_0+w_1+w_2)^2 = 9w_0w_1w_2$
are precisely the triples $\bigl(k_0^2,k_1^2,k_2^2\bigr)$,
where $(k_0,k_1,k_2)$ is a Markov triple.
\end{lem}

\begin{proof}
Clearly, every squared Markov triple solves the equation.
For the converse, we build up the analogue of the
Markov tree. Consider the involution
\[
\lambda \colon \ (u_0,u_1,u_2)
 \mapsto
\bigl(u_0, u_1, 9u_0u_1-6\sqrt{u_0u_1u_2}+u_2\bigr)
 =
\bigl(u_0, u_1, \bigl(3\sqrt{u_0u_1}-\sqrt{u_2}\bigr)^2\bigr).
\]
If $u$ is a positive integer solution, then also $\lambda(u)$
is one. Moreover, if the entries of~$u$ are squares,
then the entries of $\lambda(u)$ are so.
Starting with $(1,1,1)$, we obtain
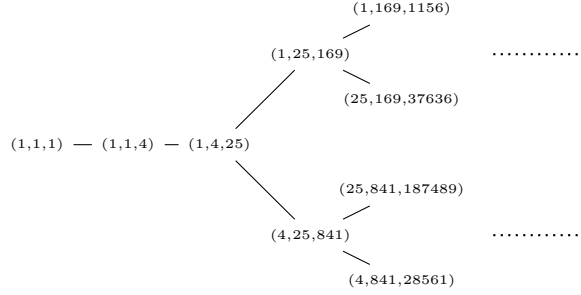
\[
\begin{tikzpicture}[scale=0.6]
\sffamily
\node[] (111) at (0,0) {$\scriptscriptstyle (1,1,1)$};
\node[] (112) at (2,0) {$\scriptscriptstyle (1,1,4)$};
\node[] (125) at (4,0) {$\scriptscriptstyle (1,4,25)$};
\node[] (1513) at (6,2) {$\scriptscriptstyle (1,25,169)$};
\node[] (2529) at (6,-2) {$\scriptscriptstyle (4,25,841)$};
\node[] (113134) at (8,3) {$\scriptscriptstyle (1,169,1156)$};
\node[] (513194) at (8,1) {$\scriptscriptstyle (25,169,37636)$};
\node[] (529533) at (8,-1) {$\scriptscriptstyle (25,841,187489)$};
\node[] (229169) at (8,-3) {$\scriptscriptstyle (4,841,28561)$};
\draw[] (111) edge (112);
\draw[] (112) edge (125);
\draw[] (125) edge (1513);
\draw[] (125) edge (2529);
\draw[] (1513) edge (113134);
\draw[] (1513) edge (513194);
\draw[] (2529) edge (529533);
\draw[] (2529) edge (229169);
\draw[thick, dotted] (10,2) edge (12,2);
\draw[thick, dotted] (10,-2) edge (12,-2);
\end{tikzpicture}
\]
by successively applying $\lambda$ to permutations of
triples obtained so far.
One directly checks that this yields the squared triples
of the Markov tree. We claim
\[
u_0 \le u_1 \le u_2, \ u_2 \ge 3
\ \implies \
(3\sqrt{u_0u_1}-\sqrt{u_2})^2
 < u_2
\]
for any positive integer solution $u = (u_0,u_1,u_2)$.
Suppose that we have ``$\ge$'' on the right hand side.
Then $9u_0^2u_1^2 \ge 4u_0u_1u_2$ and we obtain
\[
(u_0+u_1+u_2)^2
 =
9u_0u_1u_2
 \ge
4u_2^2.
\]
Consequently $u_0+u_1 \ge u_2$. This in turn gives
us $2(u_0+u_1) \ge u_0+u_1+u_2$ and the claim directly
follows from the estimate
\[
\frac{3}{u_2} + \frac{1}{u_0}
 \ge
\frac{u_0}{u_1u_2} + \frac{2}{u_2} + \frac{u_1}{u_0u_2}
 =
\frac{(u_0+u_1)^2}{u_0u_1u_2}
 \ge
\frac{1}{4} \frac{(u_0+u_1+u_2)^2}{u_0u_1u_2}
 =
\frac{9}{4}.
\]
We conclude that every positive integer solution
of $(w_0+w_1+w_2)^2 = 9w_0w_1w_2$ arises from
$(1,1,1)$ by successively applying $\lambda$ to
permutations of triples.
\end{proof}

\begin{proof}[Proof of Proposition~\ref{prop:toricmarkovchar}]
Consider $Z = Z(P)$ as in Construction~\ref{constr:fwps}.
Then, with $w=w(P)$, Proposition~\ref{prop:fwppKsquare}
tells us
\[
\mathcal{K}_Z^2
 =
\frac{(w_0+w_1+w_2)^2}{w_0w_1w_2}.
\]
The implication ``(i) $\Rightarrow$ (ii)''
is a direct consequence.
For the reverse direction, we additionally
use Lemma~\ref{lem:diophantic}.
\end{proof}

We take a brief look at the singularities of
the toric Markov surfaces; we refer
to~\cite[Sections~2 and~4]{HaPro} for a
comprehensive, more general treatment.
Let $k,p$ be coprime positive integers, denote
by \smash{$C\bigl(k^2\bigr) \subseteq \CC^*$} the group of $k^2$-th
roots of unity and consider the action
\[
C\bigl(k^2\bigr) \times \CC^2 \to \CC^2,
\qquad
\zeta \cdot z = \bigl(\zeta z_1, \zeta^{pk-1} z_2\bigr).
\]
Then \smash{$U := \CC^2/C\bigl(k^2\bigr)$} is an affine toric surface
and the image $u \in U$ of $0 \in \CC^2$ is singular
as soon as $k>1$.
A \emph{singularity of type} $\frac{1}{k^2}(1,pk-1)$
is a surface singularity locally isomorphic to $u \in U$.
The \emph{local Gorenstein index} $\iota(x)$ of a point
$x$ in a normal variety $X$ is the order of the canonical
class in the local class group $\Cl(X,x)$.

\begin{prop}
\label{prop:tmssingularities}
The fixed points $z(i) \in Z$, $i = 0,1,2$, of a toric
Markov surface $Z = \PP\bigl(k_0^2,k_1^2,k_2^2\bigr)$ are of
local Gorenstein index $k_i$ and singularity type
\smash{$\frac{1}{k_i^2}(1,p_ik_i-1)$}.
\end{prop}

\begin{lem}
\label{label:tsing}
Consider an affine toric surface $U$ with fixed
point $u \in U$.
Then $u \in U$ is of type \smash{$\frac{1}{k^2}(1,pk-1)$}
if and only if $\cl(u) = \iota(u)^2$.
\end{lem}

\begin{proof}
Let $u \in U$ be of type \smash{$\frac{1}{k^2}(1,pk-1)$}.
Choose $a,b \in \ZZ$ such that $ak - bp = 1$.
Consider the affine toric surface $U'$ given by the
generator matrix
\[
P'
 =
\begin{bmatrix}
k & k
\\
k + b & b
\end{bmatrix}.
\]
The corresponding homomorphism $\TT^2 \to \TT^2$,
\smash{$t \mapsto \bigl(t_1^kt_2^k, t_1^{k+b}t_2^b\bigr)$}
extends to a toric morphism $\pi \colon \CC^2 \to U'$.
Moreover, we obtain an isomorphism
\[
C\bigl(k^2\bigr) \to \ker(\pi),
\qquad
\zeta \mapsto \bigl(\zeta, \zeta^{pk-1}\bigr).
\]
We conclude $U' \cong \CC^2/C\bigl(k^2\bigr)$ with $C\bigl(k^2\bigr)$ acting
as needed for type $\frac{1}{k^2}(1,pk-1)$.
Thus $U' \cong U$ and, using \cite[Remark~3.7]{HaHaSp},
we obtain
\[
\cl(u) = \vert {\det}(P') \vert = k^2,
\qquad
\iota(u) = k.
\]
Conversely, assume $\cl(u) = \iota(u)^2$.
The affine toric surface $U$
is given by a generator matrix $P$.
With $k := \iota(u)$, a suitable unimodular transformation
turns $P$ into
\[
P
 =
\begin{bmatrix}
k & k
\\
c & b
\end{bmatrix},
\qquad
\gcd(c,k) = \gcd(b,k)=1.
\]
By assumption, $\cl(u) = \vert {\det}(P) \vert$
equals $\iota(u)^2 = k^2$.
Thus, we may assume $c = k+b$.
Take $a,p \in \ZZ$ with
$ak - bp = 1$ and $p \ge 1$.
Then we have an action
\[
C\bigl(k^2\bigr) \times \CC^2 \to \CC^2,
\qquad
\zeta \cdot z = \bigl(\zeta z_1, \zeta^{pk-1} z_2\bigr).
\]
With similar arguments as above, we verify that
$U$ is the quotient $\CC^2/C\bigl(k^2\bigr)$ for this action
and thus see that $u \in U$ is of type
\smash{$\frac{1}{k^2}(1,pk-1)$}.
\end{proof}

\begin{proof}[Proof of Proposition~\ref{prop:tmssingularities}]
We have $\Cl(Z) = \ZZ$ and the anticanonical class of $Z$ is
given by $w_Z = k_0^2+k_1^2+k_2^2 = 3k_0k_1k_2 \in \ZZ$.
Remark~\ref{rem:fakewv} tells us $\Cl(Z,z(i)) = \ZZ / k_i^2\ZZ$.
As Markov numbers are pairwise coprime and not divisible by 3,
we see that $w_Z$ is of order $k_i$ in $\Cl(Z,z(i))$.
The~assertion follows from Lemma~\ref{label:tsing}.
\end{proof}

\section[Rational projective C*-surfaces]{Rational projective $\boldsymbol{\CC^*}$-surfaces}

We first recall the necessary theory of quasismooth,
rational, projective $\CC^*$-surfaces of Picard number
one; see~\cite[Section~5.4]{ArDeHaLa} and the introductory
part of~\cite{HaHaHaSp} for the general background.
Then, in Construction~\ref{constr:covs}, we exhibit
for each of our $\CC^*$-surfaces two coverings onto
fake weighted projective planes.
Moreover, in Construction~\ref{constr:degs} and
Proposition~\ref{prop:degprops}, we take an explicit
look at the toric degenerations.

A point of a rational
$\CC^*$-surface $X$ is called \emph{quasismooth}
if it is the image of a smooth point of the
characteristic space $\hat X$ over $X$;
see~\cite[Section~5]{HaHaHaSp}.
It is a specific feature of a rational
$\CC^*$-surface that its singular quasismooth
points are precisely its cyclic quotient
singularities; see~\cite[Corollary~6.12]{HaHu}.

\begin{constr}[Quasismooth $\CC^*$-surfaces of Picard number one]
\label{constr:cstarsurf}
Consider an integral $3 \times 4$ matrix of the following shape:
\begin{gather*}
P
 =
\begin{bmatrix}
-1 & -1 & l_1 & 0
\\
-1 & -1 & 0 & l_2
\\
0 & d_0 & d_1 & d_2
\end{bmatrix},
\\
1 \le d_1 \le l_1 \le l_2, \qquad \gcd(l_i,d_i) = 1,
\qquad
d_0 + \frac{d_1}{l_1} + \frac{d_2}{l_2} < 0 < \frac{d_1}{l_1} + \frac{d_2}{l_2}.
\end{gather*}
Let $Z(P)$ denote the fake weighted projective space
having $P$ as its generator matrix, see Construction~\ref{constr:fwps}.
Then we obtain a surface
\[
X(P) := \overline{V(1+S_1+S_2)} \subseteq Z(P),
\]
where $S_1$, $S_2$, $S_3$ are the coordinates on the acting
torus $\TT^3 \subseteq Z$.
The surface $X(P)$ inherits from $Z(P)$ the $\CC^*$-action
given on $\TT^3 \subseteq Z(P)$ by
\[
t \cdot s = (s_1,s_2,ts_3).
\]
\end{constr}

\begin{prop}
\label{prop:cstarsurf}
Consider $X=X(P)$ in $Z=Z(P)$ given by
Construction~$\ref{constr:cstarsurf}$.
Then, in homogeneous coordinates on $Z$,
we have the representation
\[
X
 =
V\bigl(T_1T_2 + T_3^{l_1} + T_4^{l_2}\bigr)
 \subseteq
Z.
\]
The $\CC^*$-surface $X$ is projective, rational,
quasismooth, del Pezzo and of Picard number one.
With any $l_1$-th root $\zeta$ of $-1$, the $\CC^*$-fixed
points of $X$ are
\[
x_0 = [0,0,\zeta,1],
\qquad
x_1 = [0,1,0,0],
\qquad
x_2 = [1,0,0,0].
\]
The fixed point $x_0$ is hyperbolic and $x_1$, $x_2$ are both
elliptic.
There are exactly two non-trivial orbits $\CC^* \cdot z_1$
and $\CC^* \cdot z_2$ with non-trivial isotropy groups:
\[
z_1 = [-1,1,0,1], \quad
\vert \CC^*_{z_1} \vert = l_1,
\qquad
z_2 = [-1,1,1,0], \quad
\vert \CC^*_{z_2} \vert = l_2.
\]
The fake weight vector $w(P) =(w_1,w_2,w_3,w_4)$
of the ambient fake weighted projective space $Z = Z(P)$
is given explicitly in terms of $P$ as
\[
w(P)
 =
(-l_1l_2d_0 - l_2d_1 - l_1d_2, l_2d_1 + l_1d_2, - l_2d_0, -l_1d_0)
 \in
\ZZ_{>0}^4.
\]
Moreover, for the local class group orders of the three
$\CC^*$-fixed points $x_0,x_1,x_2 \in X$, we obtain
\[
\cl(x_0) = -d_0,
\qquad
\cl(x_1) = w_2,
\qquad
\cl(x_2) = w_1.
\]
Finally, the self intersection number of the canonical
divisor $\mathcal{K}_X$ on $X$ can be expressed as
follows:
\[
\mathcal{K}_X^2
 =
\biggl( \frac{1}{w_1} + \frac{1}{w_2} \biggr)
\biggl(2 + \frac{l_1}{l_2} + \frac{l_2}{l_1} \biggr)
 =
\frac{\cl(x_0)}{\cl(x_1)\cl(x_2)} (l_1+l_2)^2.
\]
\end{prop}

\begin{proof}[Proof of Construction~\ref{constr:cstarsurf}
and Proposition~\ref{prop:cstarsurf}]
The assumptions on $l_i$, $d_i$ made in
Construction~\ref{constr:cstarsurf}
ensure that~$P$ fits into the setting of
\cite[Construction~4.2]{HaHaHaSp}.
According to~ \cite[Proposition~4.5]{HaHaHaSp}, the
output $X(P)$ is a normal, rational,
projective $\CC^*$-surface.
Quasismoothness, $\rho(X)=1$ and the
statements on the fixed points are covered
by \cite[Propositions~4.9, 4.15 and~5.1]{HaHaHaSp}.

We are left with the canonical self intersection number.
Using the general formula~\cite[Proposition~7.9]{HaHaSp}
for rational projective $\CC^*$-surfaces, we directly
compute
\[
\mathcal{K}_X^2
 =
\frac{\bigl(\frac{1}{l_1}+\frac{1}{l_2}\bigr)^2}{\frac{d_1}{l_1}+\frac{d_2}{l_2}}
 -
\frac{\bigl(\frac{1}{l_1}+\frac{1}{l_2}\bigr)^2}{d_0+\frac{d_1}{l_1}+\frac{d_2}{l_2}}
 =
\biggl(
\frac{1}{\cl(x_1)}
 +
\frac{1}{\cl(x_2)}
\biggr)
\biggl(2+\frac{l_1}{l_2}+\frac{l_2}{l_1}\biggr),
\]
where $\cl(x_i)$ are the local class group orders
of the fixed points as just determined.
The assertion then follows from
$\cl(x_1)+\cl(x_2)= l_1l_2\cl(x_0)$.
\end{proof}

\begin{prop}
\label{prop:allqs}
Let $X$ be a non-toric, quasismooth, rational, projective $\CC^*$-surface
of Picard number one.
Then $X \cong X(P)$ with $P$ as in Construction~$\ref{constr:cstarsurf}$.
\end{prop}

\begin{proof}
By \cite[Theorem~4.18]{HaHaHaSp}, we have $X=X(P)$ with a
defining matrix $P$ in the sense of~\cite[Construction~4.2]{HaHaHaSp}.
Using \cite[Propositions~4.9, 4.15 and~5.1]{HaHaHaSp}, we see that
$P$ is as in Construction~\ref{constr:cstarsurf}.
\end{proof}

\begin{constr}[coverings onto fake weighted projective planes]
\label{constr:covs}
Let $X \!=\! X(P)$ in ${Z\!=\!Z(P)}$ arise via
Construction~\ref{constr:cstarsurf} from
a matrix
\[
P
 =
\begin{bmatrix}
-1 & -1 & l_1 & 0
\\
-1 & -1 & 0 & l_2
\\
\hphantom{-}0 & \hphantom{-}d_0 & d_1 & d_2
\end{bmatrix}.
\]
Let $\Sigma$ denote the unique fan in $\ZZ^3$ having $P$
as a generator matrix
and define $\Sigma' \subseteq \Sigma$ to be the subfan
with the maximal cones
\[
\sigma_1 := \cone(v_1,v_3,v_4),
\qquad
\sigma_2 := \cone(v_2,v_3,v_4),
\qquad
\tau := \cone(v_1,v_2).
\]
Then the open toric subvariety $Z' \subseteq Z$ given
by the subfan $\Sigma' \subseteq \Sigma$ satisfies
$X \subseteq Z'$.
Set $\ell := \gcd(l_1,l_2)$ and $\ell_i := l_i/\ell$.
Consider
\[
P_1
 :=
\begin{bmatrix}
-1 & -1 & \ell_2
\\
\hphantom{-}0 & d_0l_1 & \ell_2d_1 + \ell_1d_2
\end{bmatrix},
\qquad
P_2
 :=
\begin{bmatrix}
-1 & -1 & \ell_1
\\
\hphantom{-}0 & d_0l_2 & \ell_2d_1 + \ell_1d_2
\end{bmatrix}.
\]
These are generator matrices for fake weighted
projective planes $Z_1$ and $Z_2$.
In terms of the fake weight vector
$w(P) = (w_1,w_2,w_3,w_4)$ of $Z=Z(P)$, we have
\[
w(P_1) = \bigl(\ell^{-1}w_1, \ell^{-1}w_2, w_3\bigr),
\qquad
w(P_2) = \bigl(\ell^{-1}w_1, \ell^{-1}w_2, w_4\bigr)
\]
for the respective fake weight vectors.
Let $\varphi_i \colon Z' \to Z_i$ be the toric
morphisms defined by the linear maps $
F_i \colon \ZZ^3 \to \ZZ^2$
with the representing matrices
\[
F_1
 :=
\begin{bmatrix}
 0 & 1 & 0
\\
-d_1 & d_1 & l_1
\end{bmatrix},
\qquad
F_2
 :=
\begin{bmatrix}
1 & 0 & 0
\\
d_2 & -d_2 & l_2
\end{bmatrix}.
\]
Restricting to $X \subseteq Z'$ gives a finite
covering $\varphi_1 \colon X \to Z_1$ of degree $l_1$
and a finite covering $\varphi_2 \colon X \to Z_2$
of degree $l_2$.
\end{constr}

\begin{proof}
Everything is basic toric geometry except the statement
on $\varphi_i \colon X \to Z_i$.
On the acting tori $\TT^3 \subseteq Z'$ and $\TT^2 \subseteq Z_i$,
the map $\varphi_2 \colon Z' \to Z_2$ is given by
\[
\varphi_2(s_1,s_2,s_3) = \bigl( s_1, s_2^{-d_2}s_3^{l_2} \bigr).
\]
The points of $X \cap \TT^3$ are of the form
$\xi = (\xi_1, -1-\xi_1, \xi_2)$ with
$\xi_1, \xi_2 \in \CC^*$ such that $\xi_1 \ne -1$.
For the image and the fibers, we obtain
\[
\varphi_2\bigl(X \cap \TT^3\bigr)
=
\bigl\{\eta \in \TT^2; \eta_1 \ne -1\bigr\},
\qquad
\varphi_2^{-1}(\varphi_2(\xi))
=
\bigl\{(\xi_1, -1-\xi_1, \zeta \xi_2); \zeta^{l_2} = 1\bigr\}.
\]
Consequently, $\varphi_2$ is dominant, hence surjective
and its general fiber contains precisely $l_2$ points.
With the coordinate divisors $C_1,C_2,C_3 \subseteq Z_1$,
we have
\[
Z_2 \setminus \varphi_2\bigl(X \cap \TT^3\bigr)
 =
C_1 \cup C_2 \cup C_3 \cup C_4,
\qquad
C_4
 :=
\overline{\bigl\{\eta \in \TT^2; \ \eta_1 = -1\bigr\}}
 \subseteq
Z_2.
\]
Let $D_i \subseteq X$ be the prime divisors
obtained by cutting down the coordinate divisors
of $Z$; see~\mbox{\cite[Proposition~4.9]{HaHaHaSp}}.
Using surjectivity of $\varphi_2$, we see
\[
Z_2 \setminus \varphi_2\bigl(X \cap \TT^3\bigr)
 =
\varphi_2\bigl(X \setminus \TT^3\bigr)
 =
\varphi_2(D_1) \cup \dots \cup \varphi_2(D_4).
\]
Thus, $\varphi_2 \colon X \to Z_2$ must have
finite fibers, proving everything we need.
The map $\varphi_1 \colon X \to Z_1$ can be
treated in an analogous manner.
\end{proof}

\begin{constr}[degenerations to fake weighted projective planes]
\label{constr:degs}
Consider $X=X(P)$ in $Z=Z(P)$ as provided by
Construction~\ref{constr:cstarsurf}
and set
\begin{gather*}
\mathcal{X}_1
 :=
V\bigl(T_1T_2 + ST_3^{l_1} + T_4^{l_2}\bigr)
 \subseteq
Z \times \CC,
\\
\mathcal{X}_2
 :=
V\bigl(T_1T_2 + T_3^{l_1} + ST_4^{l_2}\bigr)
 \subseteq
Z \times \CC,
\end{gather*}
where the $T_i$ are the homogeneous coordinates
on $Z$ and $S$ is the coordinate on $\CC$.
Then $\mathcal{X}_1$ and~$\mathcal{X}_2$
are invariant under the respective $\CC^*$-actions
on $Z \times \CC$ given by
\begin{gather*}
\vartheta \cdot ([z_1,z_2,z_3,z_4],s)
 =
\bigl(\bigl[z_1,z_2,\vartheta^{-1}z_3,z_4\bigr],\vartheta s\bigr),
\\
\vartheta \cdot ([z_1,z_2,z_3,z_4],s)
 =
\bigl(\bigl[z_1,z_2,z_3,\vartheta^{-1}z_4r],\vartheta s\bigr).
\end{gather*}
Restricting the projection $Z \times \CC \to \CC$
yields flat families
$\psi_i \colon \mathcal{X}_i \to \CC$
being compatible with the above $\CC^*$-actions
and the scalar multiplication on $\CC$.
Set
\[
\tilde P_1
 :=
\begin{bmatrix}
d_1 & d_1+l_1d_0 & d_2
\\
l_1 & l_1 & -l_2
\end{bmatrix},
\qquad
\tilde P_2
 :=
\begin{bmatrix}
d_2 & d_2+l_2d_0 & d_1
\\
l_2 & l_2 & -l_1
\end{bmatrix},
\]
and let $\tilde Z_1$, $\tilde Z_2$ denote
the associated fake weighted projective planes.
Then the central fiber $\psi_i^{-1}(0)$
equals $\tilde Z_i$ and any other fiber
$\psi_i^{-1}(s)$ is isomorphic to $X$.
\end{constr}

\begin{proof}
The families $\psi_i \colon \mathcal{X}_i \to \CC$
are those provided by~\cite[Construction~4.1]{HaHaSu}
for $\kappa = 1,2$
and from~\cite[Proposition~4.6]{HaHaSu} we infer that
the $\tilde P_i$ are the generator matrices of
the central fibers.
Thus, $\psi_i^{-1}(0) = \tilde Z_i$.
\end{proof}

\begin{prop}
\label{prop:degprops}
Consider $X=X(P)$ in $Z=Z(P)$ from
Construction~$\ref{constr:cstarsurf}$
and the families $\mathcal{X}_i \to \CC$
from Construction~$\ref{constr:degs}$.
With $w(P) = (w_1,w_2,w_3,w_4)$,
we have
\[
w\bigl(\tilde P_1\bigr)
 =
\bigl(w_1, w_2, -l_1^2d_0\bigr),
\qquad
w\bigl(\tilde P_2\bigr)
 =
\bigl(w_1, w_2, -l_2^2d_0\bigr)
\]
for the fake weight vectors of the central
fibers $\tilde Z_1$ and $\tilde Z_2$.
Moreover, the canonical self intersection
numbers of $X$, $\tilde Z_1$, $\tilde Z_2$ satisfy
\[
\mathcal{K}_X^2
 =
\mathcal{K}_{\tilde Z_1}^2
 =
\mathcal{K}_{\tilde Z_2}^2.
\]
\end{prop}

\begin{proof}
The statement on the fake weight vectors is
obtained by direct computation.
Also the identity of the canonical self
intersections can be directly verified,
using Propositions~\ref{prop:fwppKsquare}
and~\ref{prop:cstarsurf}.
\end{proof}

\begin{rem}
The flat families $\psi_i \colon \mathcal{X}_i \to \CC$
from Construction~\ref{constr:degs} are special equivariant
test configurations in the sense of~\cite[Definition~5.2]{HaHaSu}
for the del Pezzo $\CC^*$-surface $X$.
Moreover, according to~\cite[Proposition~5.4]{HaHaSu},
any other special equivariant test configuration of
$X$ has limit $\tilde Z_1$ or $\tilde Z_2$.
\end{rem}


\section[Markov C\^{}*-surfaces]{Markov $\boldsymbol{\CC^*}$-surfaces}

In Construction~\ref{prop:markovchar}, we associate
with each pair of adjacent Markov triples
a rational, projective $\CC^*$-surface.
Theorem~\ref{prop:markovsurface} gathers geometric
properties of these \emph{Markov $\CC^*$-surfaces},
showing in particular that they naturally represent
the edges of the Markov graph.
Theorem~\ref{thm:markovsurfchar}
characterizes the Markov $\CC^*$-surfaces
as the rational, projective $\CC^*$-surfaces
of Picard number one of canonical self intersection
nine.
We begin with a couple of elementary observations
around Markov triples.

\begin{lem}
\label{lem:markovchar}
For any $0 \le k_1 \le k_2 \in \RR$ and $0 \le l_1 \le l_2 \in \RR$,
the following three conditions are equivalent:
\begin{enumerate}\itemsep=0pt
\item[$(i)$]
$l_1l_2 = k_1^2 + k_2^2$ and $l_1+l_2 = 3k_1k_2$,
\item[$(ii)$]
$l_1^2 + k_1^2 + k_2^2 = 3l_1k_1k_2$ and $l_2^2 + k_1^2 + k_2^2 = 3l_2k_1k_2$,
\item[$(iii)$]
$l_1$, $l_2$ are given in terms of $k_1$, $k_2$ as
\[
l_{1,2} = \frac{3k_1k_2 \pm \sqrt{9k_1^2k_2^2 - 4k_1^2 - 4k_2^2}}{2}.
\]
\end{enumerate}
In particular, given any two real numbers $0 \le k_1 \le k_2$,
there exist unique real numbers $0 \le l_1 \le l_2$
satisfying Condition~$(i)$.
\end{lem}

\begin{proof}
Suppose that (i) holds. Then the second equation gives
$l_2 = 3k_1k_2 - l_1$. Plugging this into the first
equation, we obtain the first equation of~(ii).
Similarly, considering $l_1 = 3k_1k_2 - l_2$, we arrive
at the second equation of~(ii).
Solving the equations of~(ii) for $l_1$ and $l_2$
gives~(iii).
If~(iii) holds, then we directly compute
$l_1l_2 = k_1^2 + k_2^2$ and $l_1+l_2 = 3k_1k_2$.
\end{proof}

\begin{lem}
\label{prop:markovchar}
Let $k_1 \le k_2$ and $l_1 \le l_2$ be positive integers.
Then the following statements are equivalent:
\begin{enumerate}\itemsep=0pt
\item[$(i)$]
$l_1l_2 = k_1^2 + k_2^2$ and $l_1+l_2 = 3k_1k_2$,
\item[$(ii)$]
$l_1^2 + k_1^2 + k_2^2 = 3l_1k_1k_2$ and $l_2^2 + k_1^2 + k_2^2 = 3l_2k_1k_2$,
\item[$(iii)$]
$(l_1,k_1,k_2)$ and $(k_1,k_2,l_2)$ are Markov triples.
\end{enumerate}
If one of $(i)$ to $(iii)$ holds, then $\gcd(l_1,l_2)=1$
and the triples $(l_1,k_1,k_2)$, $(k_1,k_2,l_2)$
are adjacent and normalized up to switching $(l_1,k_1)$
in the first one.
\end{lem}

\begin{proof}
The equivalence of (i) and (ii) holds by Lemma~\ref{lem:markovchar}
and the equivalence of~(ii) and (iii) is valid by the definition of
a Markov triple.
Now assume that one of the three conditions holds.
Then $k_1$, $k_2$, $l_2$ form a Markov triple, hence are pairwise
coprime and, using $l_2 = 3k_1k_2 - l_1$, we see that
$l_1$ and $l_2$ must be coprime as well.
Moreover, $k_1 \le k_2$ and $l_1 \le l_2$ together with
$l_1+l_2=3k_1k_2$ imply that the triples are normalized
up to switching $(l_1,k_1)$ in the first one.
Finally, $l_2= 3k_1k_2 - l_1$ merely means that the
triples are adjacent.
\end{proof}

\begin{lem}
\label{prop:markovd1d2}
Let $(l_1,k_1,k_2)$ and $(k_1,k_2,l_2)$ be
Markov triples.
Then there exist integers $d_1$, $d_2$ such that
\[
k_2^2 = l_2d_1 + l_1d_2,
\qquad
1 \le d_1 \le l_1.
\]
The integers $d_1$ and $d_2$ are uniquely determined by these
properties. Moreover, they satisfy $\gcd(l_i,d_i)=1$.
\end{lem}

\begin{proof}
Consider the factor ring $\ZZ / l_1\ZZ$.
Lemma~\ref{prop:markovchar} says that
$l_2$ and $l_1$ are coprime.
Consequently, there is a~multiplicative inverse
$\bar c_1 \in \ZZ / l_1\ZZ$ of $\bar l_2 \in \ZZ / l_1\ZZ$.
We claim that there is a~unique $d_1 \in \ZZ$ with
\[
1 \le d_1 \le l_1,
\qquad
\bar d_1 = \bar c_1 \cdot \bar k_2^2 \in \ZZ / l_1\ZZ.
\]
Indeed, since $(l_1,k_1,k_2)$ is a Markov triple,
the numbers $l_1$ and $k_2$ are coprime.
Thus, being a~product of units,
$\bar d_1 := \bar c_1 \cdot \bar k_2^2$
is a unit in $\ZZ / l_1\ZZ$.
We take the unique representative $1 \le d_1 \le l_1$.
Now, there is a unique $d_2 \in \ZZ$ with
\[
l_1d_2 + l_2d_1 = k_2^2,
\]
because of
$\bar l_2 \cdot \bar d_1 = \bar l_2 \cdot \bar c_1 \cdot \bar k_2^2 = \bar k_2^2$
in $\ZZ / l_1\ZZ$.
To obtain $\gcd(l_2,d_2)=1$, look at the factor ring~$\ZZ/l_2\ZZ$.
There $\bar l_1$ admits a multiplicative inverse $\bar c_2$.
This gives us $\bar d_2 = \bar c_2 \cdot \bar k_2^2$ in $\ZZ/l_2\ZZ$.
Hence~$\bar d_2$ is a unit in $\ZZ/l_2\ZZ$.
\end{proof}

\begin{constr}
\label{constr:markovsurface}
Let $\mu = ((l_1,k_1,k_2), (k_1,k_2,l_2))$
be a pair of adjacent Markov triples,
the second triple normalized
and the first one up to switching $(l_1,k_1)$.
Let $d_1, d_2 \in \ZZ$ be as provided
by Lemma~\ref{prop:markovd1d2} and set
\[
P(\mu)
 :=
\begin{bmatrix}
-1 & -1 & l_1 & 0
\\
-1 & -1 & 0 & l_2
\\
\hphantom{-} 0 & -1 & d_1 & d_2
\end{bmatrix}.
\]
As in Construction~\ref{constr:cstarsurf},
let $Z(\mu)$ be the toric threefold defined
by the complete fan in $\ZZ^3$ having~$P(\mu)$
as its generator matrix, let $S_1$, $S_2$, $S_3$ be
the coordinates of the acting torus
$\TT^3 \subseteq Z(\mu)$ and set
\[
X (\mu) := \overline{V(1+S_1+S_2)} \ \subseteq \ Z(\mu).
\]
\end{constr}

\begin{thm}
\label{prop:markovsurface}
Let $X=X(\mu)$ be as in Construction~$\ref{constr:markovsurface}$.
Then $X$ is a quasismooth, rational $\CC^*$-surface with
$\Cl(X)=\ZZ$ and Cox ring $\mathcal{R}(X)$ given~by
\begin{gather*}
\mathcal{R}(X)
 =
\CC[T_1,T_2,T_3,T_4]/\bigl\langle T_1T_2+T_3^{l_1}+T_4^{l_2}\bigr\rangle,
\\
\deg(T_1) = k_1^2,\qquad
\deg(T_2) = k_2^2,
\qquad
\deg(T_3) = l_2,
\qquad
\deg(T_4) = l_1.
\end{gather*}
The ambient toric threefold $Z=Z(\mu)$
is the weighted projective space
\smash{$\PP_{(k_1^2, k_2^2, l_2, l_1)}$}
and $X$ is the zero set of a homogeneous
equation of degree $k_1^2+k_2^2 = l_1l_2$:
\[
X
 =
V\bigl(T_1T_2+T_3^{l_1}+T_4^{l_2}\bigr)
 \subseteq
\PP_{(k_1^2, k_2^2, l_2, l_1)}
 =
Z.
\]
The surface $X$ is of Picard number one,
it is non-toric if and only if $l_1 > 1$
and the $\CC^*$-action on $X$ is given in homogeneous
coordinates by
\[
t \cdot [z] = \bigl[t \cdot z_1, t^{-1} \cdot z_2, z_3, z_4\bigr].
\]
The only possible singularities of $X$ are the
elliptic fixed points, given together with their
local class group order, local Gorenstein index
and singularity type by
\begin{gather*}
x_1 = [0,1,0,0], \qquad \cl(x_1) = k_2^2, \qquad \iota(x_1) = k_2,
\qquad \frac{1}{k_2^2}(1,p_2k_2-1),
\\
x_2 = [1,0,0,0], \qquad \cl(x_2) = k_1^2, \qquad \iota(x_2) = k_1,
\qquad \frac{1}{k_1^2}(1,p_1k_1-1),
\end{gather*}
with $p_i \in \ZZ_{\ge 1}$ such that $\gcd(p_i,k_i)=1$.
For the canonical self intersection of $X$, we have
$\mathcal{K}_X^2 = 9$.
Moreover, there is a commutative diagram
\[
\xymatrix{
&&
{\PP_{(k_1^2, k_2^2, l_2, l_1)}}
\ar@{-->}[dll]_{[z_1,z_2,z_4^{l_1}] \mapsfrom [z_1,z_2,z_3,z_4]\qquad\qquad}
\ar@{-->}[drr]^{\qquad\qquad[z_1,z_2,z_3,z_4] \mapsto [z_1,z_2,z_3^{l_2}]}
&&
\\
{\PP_{(k_1^2,k_2^2,l_1^2)}}
&&
X
\ar[u]
\ar[ll]^{\qquad 1:l_1^2}
\ar[rr]_{l_2^2:1 \qquad}
&&
{\PP_{(k_1^2,k_2^2,l_2^2)}}
}
\]
with finite coverings
\smash{$X(k_1,k_2,l_1,l_2) \to {\PP_{(k_1^2,k_2^2,l_i^2)}}$}
of degree $l_i^2$, respectively.
Finally, we obtain flat families
$\psi_i \colon \mathcal{X}_i \to \CC$,
where
\begin{gather*}
\mathcal{X}_1
 :=
V\bigl(T_1T_2 + ST_3^{l_1} + T_4^{l_2}\bigr)
 \subseteq
\PP_{(k_1^2, k_2^2, l_2, l_1)} \times \CC,
\\
\mathcal{X}_2
 :=
V\bigl(T_1T_2 + T_3^{l_1} + ST_4^{l_2}\bigr)
 \subseteq
\PP_{(k_1^2, k_2^2, l_2, l_1)} \times \CC
\end{gather*}
and the $\psi_i$ are given by restricting the projection
\smash{$\PP_{(k_1^2, k_2^2, l_2, l_1)} \times \CC \to \CC$}.
For the fibers of these families, we have
\[
\psi_i^{-1}(s) \cong X(k_1,k_2,l_1,l_2), \quad s \in \CC^*,
\qquad
\psi_i^{-1}(0) \cong {\PP_{(k_1^2,k_2^2,l_i^2)}}.
\]
\end{thm}

\begin{proof}[Proof of Construction~\ref{constr:markovsurface}
and Theorem~\ref{prop:markovsurface}]
First, we check that $P(\mu)$ fits into
Construction~\ref{constr:cstarsurf} with $d_0=-1$.
By the normalizedness assumption, we have $l_1 \le l_2$.
Lemma~\ref{prop:markovd1d2} delivers $1 \le d_1 \le l_1$
and $\gcd(l_i,d_i)=1$. Moreover,
\[
-1 + \frac{d_1}{l_1} + \frac{d_2}{l_2}
 =
\frac{-l_1l_2 + k_2^2}{l_1l_2}
 =
-\frac{k_1^2}{l_1l_2}
 <
0
 <
\frac{k_2^2}{l_1l_2}
 =
\frac{d_1}{l_1} + \frac{d_2}{l_2}.
\]
Thus, $P = P(\mu)$ is as wanted and
Proposition~\ref{prop:cstarsurf} says that
$X=X(P)$ is a quasismooth, rational,
projective $\CC^*$-surface.
For the divisor class group and the Cox ring of $X$,
we need to have an exact sequence
\[
\xymatrix{
0
\ar[r]
&
\ZZ^3
\ar[rr]^{P^*}
&&
\ZZ^4
\ar[rr]^Q
&&
\ZZ
\ar[r]
&
0.
}
\]
The transpose matrix $P^*$ is injective
and, as $l_1$ and $l_2$ are coprime,
the columns of~$P^*$ generate a~primitive
sublattice of $\ZZ^4$ and thus we have
a torsion free cokernel.
Moreover, we directly check that
$Q \cdot P^* = 0$ holds with
\[
Q = \begin{bmatrix} k_1^2 & k_2^2 & l_2 & l_1 \end{bmatrix}.
\]
Since $k_1^2$ and $k_2^2$ are coprime, $Q$ maps onto $\ZZ$.
Consequently, \cite[Propositions~4.13 and~4.16]{HaHaHaSp} yield
$\Cl(X) = \ZZ$ and the desired presentation of $\mathcal{R}(X)$.
Moreover, we can identify the ambient toric variety as
\[
Z(\mu) = \PP_{(k_1^2, k_2^2, l_2, l_1)}.
\]
The hyperbolic fixed point $x_0$ is smooth by
\cite[Proposition~5.1]{HaHaHaSp}.
Proposition~\ref{prop:cstarsurf} and
Lemmas~\ref{prop:markovchar} and~\ref{prop:markovd1d2}
give us $\cl(x_1)$ and $\cl(x_2)$.
For the local Gorenstein indices, consider
the following linear forms
\[
u_1 = \biggl[\frac{d_2-d_1}{k_2^2}, \frac{d_1-d_2}{k_2^2}, \frac{3k_1}{k_2} \biggr],
\qquad
u_2 = \biggl[\frac{d_1-d_2+l_2}{k_1^2}, \frac{d_2-d_1+l_1}{k_1^2}, - \frac{3k_2}{k_1} \biggr].
\]
Using the identities $k_1^2 = l_1l_2-l_2d_1-l_1d_2$ and $k_2^2 = l_2d_1+l_1d_2$
just established, one directly checks that the linear forms $u_1$, $u_2$ evaluate
on the columns $v_1$, $v_2$, $v_3$, $v_4$ of~$P$ as follows:
\begin{gather*}
\bangle{u_1,v_1} = 0, \qquad
\bangle{u_1,v_3}
= 1, \qquad
\bangle{u_1,v_4}
= 1,
\\
\bangle{u_2,v_2} = 0, \qquad
\bangle{u_2,v_3}
= 1, \qquad
\bangle{u_2,v_4}
= 1.
\end{gather*}
We can conclude that $k_2u_1$ and $k_1u_2$ are primitive
integral vectors and~\cite[Proposition~8.9]{HaHaSp} yields
$\iota(x_1)=k_2$ and $\iota(x_2)=k_1$.
As~$X$ is quasismooth, $x_1$, $x_2$ are toric
singularities and Lemma~\ref{label:tsing}
gives us their singularity type.
We obtain
\[
\mathcal{K}_X^2
 =
\frac{\cl(x_0)}{\cl(x_1)\cl(x_2)} (l_1+l_2)^2
 =
\biggl(\frac{l_1+l_2}{k_1k_2}\biggr)^2
 =
9
\]
for the canonical self intersection number,
using Proposition~\ref{prop:cstarsurf} and
Lemma~\ref{prop:markovchar} for the last equation.
The desired coverings from $X$ onto the
weighted projective planes are those from
Construction~\ref{constr:covs} followed by the
obvious ones:
\[
\PP_{(k_1^2,k_2^2,l_i)} \to \PP_{(k_1^2,k_2^2,l_i^2)},
\qquad
[z_1,z_2,z_3] \mapsto \bigl[z_1,z_2,z_3^{l_i}\bigr].
\]
Finally, the families $\mathcal{X}_i \to \CC$
are provided by Construction~\ref{constr:degs}
and their properties claimed in the assertion
are guaranteed by
Proposition~\ref{prop:degprops}.
\end{proof}


\section{Proof of the main results}

Here we prove the main results of this note,
Theorems~\ref{thm:markovsurfchar} and \ref{cor:planedegs}.
A first observation is that log terminal $\CC^*$-surfaces
of Picard number one and canonical self intersection nine
are quasismooth.

\begin{prop}
\label{prop:notqs}
Let $X$ be a log terminal, rational, projective
$\CC^*$-surface of Picard number one with
$\mathcal{K}_X^2 = 9$. Then $X$ is quasismooth.
\end{prop}

\begin{proof}
We use the description of rational, projective
$\CC^*$-surfaces of Picard number one via
defining matrices $P$ provided
by~\cite[Construction~4.2, Proposition~4.5]{HaHaHaSp}.
Log terminality of $X$ imposes strong
conditions on the upper rows of $P$,
see~\cite[Proposition~5.9]{HaHaHaSp}.

The non-quasismooth log terminal surface singularities
are those of the types $D$ or $E$, where ``type $D$''
refers to the platonic triples $(2,2,n)$ and ``type $E$''
gathers the platonic triples $(2,3,3)$, $(2,3,4)$
and $(2,3,5)$ in Brieskorn's result~\cite[Satz~2.10]{Br};
see also~\cite[Proposition~8.14]{HaHaSp}.
The first tuple of possible upper entries
from~\cite[Proposition~5.9]{HaHaHaSp} is
$(1,y,2,2)$ and the associated defining
matrix is of the form
\begin{gather*}
P
 =
\begin{bmatrix}
-1 & -l_0 & 2 & 0
\\
-1 & -l_0 & 0 & 2
\\
\hphantom{-} 0 & \hphantom{-}d_0 & d_1 & 1
\end{bmatrix},
\qquad
\gcd(l_0,d_0) = \gcd(2,d_1) = 1,
\\
d_0 < 0, \qquad
l_0 \ge 2, \qquad
\frac{d_1}{2}+\frac{1}{2} > 0, \qquad
\frac{d_0}{l_0}+\frac{d_1}{2}+\frac{1}{2} < 0.
\end{gather*}
The resulting $\CC^*$-surface $X$, built as in
Construction~\ref{constr:cstarsurf}, has
$[1,0,0,0] \in X$ as singular point
of type $D$.
From~\cite[Proposition~7.9]{HaHaSp}, we infer
\begin{align*}
\mathcal{K}_X^2
&{}=
\frac{1}{\frac{d_1}{2}+\frac{1}{2}}
 -
\frac{2-l_0-\frac{1}{l_0}}{d_0}
 -
\frac{1}{l_0^2\bigl(\frac{d_0}{l_0}+\frac{d_1}{2}+\frac{1}{2}\bigr)}
\\[10pt]
&{}=
\frac{2}{d_1+1}
 +
\frac{(l_0-1)^2}{d_0l_0}
 -
\frac{2}{l_0^2(d_1l_0 + 2d_0 + l_0)}.
\end{align*}
Since the second right-hand side term is negative,
we have $\mathcal{K}_X^2 < 4$ in this case.
The next tuple of upper entries is $(1,2,y,2)$,
which leads to the defining matrix
\begin{gather*}
P
 =
\begin{bmatrix}
-1 & -2 & l_1 & 0
\\
-1 & -2 & 0 & 2
\\
\hphantom{-} 0 & \hphantom{-}d_0 & d_1 & 1
\end{bmatrix},
\qquad
\gcd(2,d_0) = \gcd(l_1,d_1) = 1,
\\
d_0 < 0, \qquad l_1 \ge 2, \qquad \frac{d_1}{l_1}+\frac{1}{2} > 0, \qquad \frac{d_0}{2}+\frac{d_1}{l_1}+\frac{1}{2} < 0.
\end{gather*}
As before, the resulting $\CC^*$-surface $X$
has $[1,0,0,0] \in X$ as a singular point
of type~$D$.
This time, the canonical self intersection
is the following:
\[
\mathcal{K}_X^2
 =
\frac{\bigl(\frac{1}{l_1}+\frac{1}{2}\bigr)^2}{\frac{d_1}{l_1}+\frac{1}{2}}
 +
\frac{1}{2d_0}
 -
\frac{1}{l_1^2\bigl(d_0+\frac{d_1}{l_1}+\frac{1}{2}\bigr)}.
\]
Note that the first term is not bounded from above,
so we can't estimate $\mathcal{K}_X^2$.
Instead we observe that $\mathcal{K}_X^2=9$ is a
quadratic equation in $d_0$ with discriminant
\[
\Delta
 =
36d_1^2 + 36d_1l_1 + 9l_1^2 - 8d_1 - 4l_1.
\]
The key observations are that $\Delta$ factors as
$\Delta = a(9a-4)$ for $a=2d_1+l_1>0$ and
that $a(9a-4)$ never is a square.
Thus we can conclude $\mathcal{K}_X^2 \ne 9$.

These two sample cases show all the arguments
for the remaining ones: either we directly
estimate $\mathcal{K}_X^2 < 9$ or we can
show that $\mathcal{K}_X^2 = 9$ admits no
integral solution.
Concerning the latter, $(1,2,y,2)$ is in fact
the most tricky case.
\end{proof}

\begin{rem}
Log terminality is essential in Proposition~\ref{prop:notqs}.
For instance, consider the rational, projective
$\CC^*$-surfaces $X$ built from the matrices
\[
P
 =
\begin{bmatrix}
-1 & -l^2 + 4l - 1 & l^2 - l + 1 & 0
\\
-1 & -l^2 + 4l - 1 & 0 & 2
\\
\hphantom{-} 0 & -1 & \frac{l-l^2}{2} & 1
\end{bmatrix},
\qquad l \ge 5,
\]
exactly as in Construction~\ref{constr:cstarsurf}.
Then $\rho(X)=1$ and $\mathcal{K}_X^2=9$ by~\cite[Proposition~6.9]{HaHaSp}.
From \mbox{\cite[Corollary~8.12]{HaHaSp}}, we infer that $X$ is
not log terminal, thus not quasismooth.
\end{rem}

\begin{defi}
By a \emph{Markov $\CC^*$-surface} we mean
a $\CC^*$-surface isomorphic to some $X(\mu)$
arising from Construction~\ref{constr:markovsurface}.
\end{defi}

\begin{thm}
\label{thm:markovsurfchar}
Let $X$ be a non-toric, log terminal,
rational $\CC^*$-surface of Picard number one.
Then the following statements are equivalent:
\begin{enumerate}\itemsep=0pt
\item[$(i)$]
$X$ is a non-toric Markov $\CC^*$-surface.
\item[$(ii)$]
We have $\mathcal{K}_X^2 = 9$.
\end{enumerate}
Moreover, if $X$ satisfies $(i)$ or $(ii)$, then it is
determined up to isomorphy by the local class group
orders $\cl(x_1)$, $\cl(x_2)$ of its elliptic fixed points.
\end{thm}

\begin{proof}
The implication ``(i) $\Rightarrow$ (ii)'' holds
by Theorem~\ref{prop:markovsurface}.
Let~(ii) be valid.
Then Proposition~\ref{prop:notqs} shows
that~$X$ is quasismooth.
Thus, Proposition~\ref{prop:allqs} allows us to
assume $X=X(P)$ with
\begin{gather*}
P
 =
\begin{bmatrix}
-1 & -1 & l_1 & 0
\\
-1 & -1 & 0 & l_2
\\
\hphantom{-} 0 & \hphantom{-}d_0 & d_1 & d_2
\end{bmatrix},
\\
2 \le l_1 \le l_2, \qquad 1 \le d_1 < l_1, \qquad \gcd(l_i,d_i) = 1,
\qquad
d_0 + \frac{d_1}{l_1} + \frac{d_2}{l_2} < 0 < \frac{d_1}{l_1} + \frac{d_2}{l_2}.
\end{gather*}
Recall from Proposition~\ref{prop:cstarsurf}
that the fake weight vector
$w(P)=(w_1,w_2,w_3,w_4)$
of the matrix~$P$ is given explicitly by
\[
w(P)
 =
(-l_1l_2d_0 - l_2d_1 - l_1d_2, l_2d_1 + l_1d_2, -d_0l_2, -d_0l_1)
 \in
\ZZ_{\ge 1}^4.
\]

We first show that $d_0=-1$ holds
and that \smash{$\bigl(w_1,w_2,l_i^2\bigr)$} is a squared
Markov triple \smash{$\bigl(k_1^2,k_2^2,l_i^2\bigr)$}.
Consider the toric degenerations \smash{$\tilde Z_1$}
and \smash{$\tilde Z_2$} of \smash{$X$} as provided by
Construction~\ref{constr:degs}.
Due to Proposition~\ref{prop:degprops},
their fake weight vectors are
$
w\bigl(\tilde P_1\bigr)
 =
\bigl(w_1, w_2, -l_1^2d_0\bigr)$,
$
w\bigl(\tilde P_2\bigr)
 =
\bigl(w_1, w_2, -l_2^2d_0\bigr)$.
Moreover, this Proposition gives us
\smash{$\mathcal{K}_{\tilde Z_i}^2 = \mathcal{K}_X^2 = 9$}.
Consequently, Proposition~\ref{prop:toricmarkovchar}
provides us with Markov triples
$(k_1,k_2,l_i\delta_0)$ such that
$
w\bigl(\tilde P_1\bigr)
 =
\bigl(k_1^2, k_2^2, l_1^2\delta_0^2\bigr)$,
$
w\bigl(\tilde P_2\bigr)
 =
\bigl(k_1^2, k_2^2, l_2^2\delta_0^2\bigr)$,
where $d_0 = - \delta_0^2$. To conclude the first step,
we have to show $\delta_0=1$.
Computing the anticanonical self intersection of $\tilde Z_i$
according to Proposition~\ref{prop:fwppKsquare}, we obtain
\[
9
 =
\frac{\bigl(w_1+w_2+l_i^2\delta_0^2\bigr)^2}{w_1 w_2 l_i^2\delta_0^2}
 =
\frac{\bigl(l_1l_2\delta_0^2+l_i^2\delta_0^2\bigr)^2}{\bigl(l_i^2\delta_0^2-w_2\bigr) w_2 l_i^2\delta_0^2}.
\]
Consequently,
\[
9\bigl(l_i^4\delta_0^2w_2 - l_i^2w_2^2\bigr)
 =
\delta_0^2\bigl(l_1l_2+l_i^2\bigr)^2.
\]
Thus, $\delta_0$ divides $3l_iw_2$.
Since the entries of the squared Markov triple
$\bigl(w_1,w_2,l_i^2\delta_0^2\bigr)$
are pairwise coprime, we have
$\gcd(l_1,l_2) = 1$ and $\delta_0 \mid 3l_i$.
Thus, $\delta_0=1$ or $\delta_0=3$.
The latter is excluded, because no Markov
number is divisible by 3.

Having two members in common, the
Markov triples $(k_1,k_2,l_i)$,
where $i=1,2$, are adjacent.
As seen above, we have $k_2^2 = l_2d_1 + l_1d_2$.
Thus, the uniqueness part of
Lemma~\ref{prop:markovd1d2} shows
that we are in the setting of
Construction~\ref{constr:markovsurface}.
The equivalence of~(i) and (ii) is proven.
The supplement follows directly from
Theorem~\ref{prop:markovsurface}
and, again, the uniqueness statement of
Lemma~\ref{prop:markovd1d2}.
\end{proof}

\begin{thm}
\label{cor:planedegs}
Let $X$ be a log terminal,
rational, projective surface with a non-trivial
torus action.
Then the following statements are equivalent:
\begin{enumerate}\itemsep=0pt
\item[$(i)$]
$X$ is a degeneration of the projective
plane.
\item[$(ii)$]
We have $\rho(X) = 1$ and $\mathcal{K}_X^2 = 9$.
\item[$(iii)$]
$X$ is a toric Markov surface or a
Markov $\CC^*$-surface.
\end{enumerate}
\end{thm}

\begin{proof}
Assume that (i) holds.
From~\cite[Theorem~4, Corollary~5]{Ma}, we infer
$\rho(X)=1$ and \smash{$K_X^2 = 9$}.
If~(ii) holds, then Proposition~\ref{prop:toricmarkovchar}
and Theorem~\ref{thm:markovsurfchar}
yield~(iii).
Now, let~(iii) be valid.
Then Proposition~\ref{prop:tmssingularities}
and
Theorem~\ref{prop:markovsurface}
ensure property (d) of Manetti's
Main Theorem~\mbox{\cite[p.~90]{Ma}},
telling us that then $X$ is
a degeneration of the plane.
\end{proof}

\begin{rem}
Note that also non-rational normal $\CC^*$-surfaces
show up as degenerations of the projective plane.
To obtain a classical example, fix $h \in \CC[T_1,T_2,T_3]$
of degree $d$ with $V(h) \subseteq \PP_2$ smooth and
consider
$
\mathcal{X}
 :=
V(T_4S - h)
 \subseteq
\PP(1,1,1,d) \times \CC$,
where $S$ is the coordinate on $\CC$. The projection
$\mathcal{X} \to \CC$ to the last coordinate has fibers
$\mathcal{X}_s \cong \PP_2$ over $s \in \CC^*$ and
the central fiber $\mathcal{X}_0$ is the cone over
the smooth curve $V(h) \subseteq \PP_2$, acted on by~$\CC^*$
via $t \cdot [z] = [z_1,z_2,z_3,tz_4]$.
\end{rem}

\subsection*{Acknowledgements}

We would like to thank Yuchen Liu for fruitful
discussions stimulating this work.
Moreover, we are grateful to Andrea Petracci for
valuable information about existing articles on
the algebraic geometry around the Markov numbers.
Finally, we are indebted to the referees for
carefully going through our manuscript, valuable
suggestions and pointing to a gap in the proof of
a~former, incorrect version of the second main
theorem.

\pdfbookmark[1]{References}{ref}
\LastPageEnding

\end{document}